\documentclass[10pt, reqno]{amsart}
\usepackage{amsmath, amsthm, amscd, amsfonts, amssymb, graphicx, color}
\usepackage[bookmarksnumbered, colorlinks, plainpages]{hyperref}
\hypersetup{colorlinks=true,linkcolor=red, anchorcolor=green, citecolor=cyan, urlcolor=red, filecolor=magenta, pdftoolbar=true}
\usepackage{mathrsfs}

\newtheorem{theorem}{Theorem}[section]
\newtheorem{lemma}[theorem]{Lemma}

\newtheorem{corollary}[theorem]{Corollary}
\theoremstyle{definition}

\theoremstyle{remark}
\newtheorem{remark}[theorem]{Remark}
\numberwithin{equation}{section}

\begin{document}
\setcounter{page}{1}

\title[Class field towers and minimal models]{Class field towers and minimal models}

\author[Nikolaev]{Igor V. Nikolaev$^1$}

\address{$^{1}$ Department of Mathematics and Computer Science, St.~John's University, 8000 Utopia Parkway,  
New York,  NY 11439, United States.} \email{\textcolor[rgb]{0.00,0.00,0.84}{igor.v.nikolaev@gmail.com}}





\begin{abstract}
We use the notion of an  Etesi $C^*$-algebra to prove that  the real class field towers are always finite.
\end{abstract}

\maketitle

\section{Introduction}
Let $k$ be a number field and let $\mathscr{H}(k)$ be the Hilbert class field of $k$, i.e. the maximal abelian unramified 
extension of $k$.  The class field tower is a sequence of the field extensions:
\begin{equation}\label{eq1.1}
k\subseteq \mathscr{H}(k)\subseteq \mathscr{H}^2(k)\subseteq \mathscr{H}^3(k)\subseteq\dots,
\end{equation}
where $\mathscr{H}^2(k)=\mathscr{H}(\mathscr{H}(k)), ~\mathscr{H}^3(k)=\mathscr{H}(\mathscr{H}^2(k))$, etc. Whether there exists an integer $m\ge 1$
such that $\mathscr{H}^m(k)\cong \mathscr{H}^{m+1}(k)\cong\dots$ is known as the  class field tower problem [Furtw\"angler 1916] \cite{Fur1}.
They say that the class field tower  is finite if $m<\infty$  and infinite otherwise.  
It is easy to see, that the tower is finite if and only if the ring of integers of the field $\mathscr{H}^m(k)$ is the principal
ideal domain, i.e.  has the class number $1$.

The Golod-Shafarevich Theorem  says that  (\ref{eq1.1}) can be infinite for some fields $k$
[Golod \& Shafarevich 1964] \cite{GolSha1}. This result solves in negative
 the class field tower problem. 
  On the other hand,  many  fields $k$  have  finite class field towers.  
     The sorting of $k$  by  the finite and infinite towers
(\ref{eq1.1})  is a difficult open problem.

In this note we study the real class field towers,
i.e. when all fields $\mathscr{H}^i(k)$ in (\ref{eq1.1})  are real. 
It is shown that  such towers are always finite, see corollary \ref{cor1.2}. 
To outline the idea,  we prove that the blow-up map  of  an algebraic surface induces 
 a Hilbert class field extension of a field coming from the $K_0$-group of the corresponding Etesi $C^*$-algebra \cite[Section 7.5]{Nik1},
 see theorem \ref{thm1.1}.  The Castelnuovo Theorem says that it takes a finite number of the blow-ups 
 of an algebraic surface to  get the minimal model, hence  (\ref{eq1.1}) is finite.
 To formalize our results, let us recall some definitions.   

An algebraic surface $S$ is a variety of complex dimension $2$.  The rational map $\phi: S\dashrightarrow S'$ is called birational, 
if its inverse $\phi^{-1}$ is a rational map. 
A birational map  $\phi: S\dashrightarrow S'$ is  a blow-up,
if it is defined everywhere except for a point $p\in S$ and a rational curve $C\subset S'$,
such  that  $\phi^{-1}(C)=p$.   Each  birational map  is composition of a finite 
number of the blow-ups.
The surface $S$ is called a minimal model, if any birational map $S\dashrightarrow S'$
is an isomorphism. 
The Castelnuovo Theorem says that $S$ is a minimal 
model if and only if $S$ does not contain rational curves $C$ with the 
self-intersection index $-1$.  In particular, the minimal model is
obtained from $S$ by a finite number of the blow-ups along $C$.

Recall that $S$ can be identified with  a smooth 4-dimensional manifold. We 
 denote by $\mathit{Diff}(S)$ 
a group of the orientation-preserving diffeomorphisms of $S$ 
and by  $\mathit{Diff}_0(S)$ the connected component of $\mathit{Diff}(S)$ 
containing the identity.  The Etesi $C^*$-algebra  $\mathbb{E}_{S}$
is  a group $C^*$-algebra   of the locally compact group 
$G:=\mathit{Diff}(S)/\mathit{Diff}_0(S)$ \cite[Definition 1.1]{Nik1}. 
Since $G$ is a countable, discrete, amenable group acting on $S$ and 
the action  admits a faithful $G$-invariant Borel probability measure (e.g. by taking
the Lebesgue measure of the orbit space $S/G$), we conclude that 
 $\mathbb{E}_{S}\cong C_0(S)\rtimes G$ embeds into a simple unital AF-algebra [Schafhauser 2020] \cite[Theorem C]{Sch1}. 
   Moreover, such an AF-algebra is stationary  \cite[Lemma  7.5.3]{Nik1} depending 
on a  constant  integer matrix 
$A\in GL(n, \mathbf{Z})$ [Blackadar 1986] \cite[Section 7.2]{B}.  
Let $\lambda_A>1$ be the Perron-Frobenius eigenvalue of $A$
and let $k=\mathbf{Q}(\lambda_A)$ be a real number 
field generated by the algebraic number $\lambda_A$. 
Our main result can be formulated  as follows.  
\begin{theorem}\label{thm1.1}
The birational map $S\dashrightarrow S'$ is a blow-up
if and only if $k'\cong \mathscr{H}(k)$. 
\end{theorem}
As explained above, theorem \ref{thm1.1} can be used to study 
 inclusions (\ref{eq1.1}). Namely, in view of  Castelnuovo's theory of the minimal models,  
 one gets the following application of theorem \ref{thm1.1}. 
\begin{corollary}\label{cor1.2}
The real class field towers are always finite.
 \end{corollary}
\begin{remark}
The  real class field towers and the class field towers over the real 
fields  are not the same in general. The latter admit imaginary Hilbert class fields 
and  can be infinite. 
\end{remark}
The article is organized as follows.  The preliminary facts can be found in Section 2. 
Theorem \ref{thm1.1} and corollary \ref{cor1.2}  are proved in Section 3.

\section{Preliminaries}
We briefly review the 4-dimensional topology,  the Etesi $C^*$-algebras and 
the algebraic surfaces.  We refer the reader to  [Beauville 1996] \cite{BE}, [Freedman \& Quinn 1990] \cite{FQ},
and \cite[Section 7.5]{Nik1} for a detailed account.

\subsection{Topology of 4-manifolds}
We denote by $\mathscr{M}$ a compact 4-dimensional manifold. 
 Unlike dimensions 2 and 3, the smooth structures are
detached from the topology of  $\mathscr{M}$. 
Due to the works of Rokhlin, Freedman and Donaldson, it is known that
 $\mathscr{M}$ can be  non-smooth and    if there exists a smooth structure, it need not be unique. 
In what follows, we assume  $\mathscr{M}$ to be a smooth 4-manifold endowed with the standard 
smooth structure, e.g. coming from  a realization of  $\mathscr{M}$ as a complex algebraic surface $S$.

Let  $\mathit{Diff}(\mathscr{M})$ be
a group of the orientation-preserving diffeomorphisms of $\mathscr{M}$
Denote by  $\mathit{Diff}_0(\mathscr{M})$  the connected component of $\mathit{Diff}(\mathscr{M})$
containing the identity.  The group  $\mathit{Diff}(\mathscr{M})/ \mathit{Diff}_0(\mathscr{M})$
is discrete and therefore locally compact.

\subsection{Etesi $C^*$-algebras}
\subsubsection{$C^*$-algebras}
The $C^*$-algebra is an algebra  $\mathscr{A}$ over $\mathbf{C}$ with a norm 
$a\mapsto ||a||$ and an involution $\{a\mapsto a^* ~|~ a\in \mathscr{A}\}$  such that $\mathscr{A}$ is
complete with  respect to the norm, and such that $||ab||\le ||a||~||b||$ and $||a^*a||=||a||^2$ for every  $a,b\in \mathscr{A}$.  
Each commutative $C^*$-algebra is  isomorphic
to the algebra $C_0(X)$ of continuous complex-valued
functions on some locally compact Hausdorff space $X$. 
Any other  algebra $\mathscr{A}$ can be thought of as  a noncommutative  
topological space. 

An  AF-algebra  (Approximately Finite-dimensional $C^*$-algebra) is
 the  norm closure of an ascending sequence of   finite dimensional
$C^*$-algebras $M_n$,  where  $M_n$ is the $C^*$-algebra of the $n\times n$ matrices
with entries in $\mathbf{C}$. Here the index $n=(n_1,\dots,n_k)$ represents
the  semi-simple matrix algebra $M_n=M_{n_1}\oplus\dots\oplus M_{n_k}$.
The ascending sequence mentioned above  can be written as 
$M_1\buildrel\rm\varphi_1\over\longrightarrow M_2
   \buildrel\rm\varphi_2\over\longrightarrow\dots,$
where $M_i$ are the finite dimensional $C^*$-algebras and
$\varphi_i$ the homomorphisms between such algebras.  
If $\varphi_i=Const$, then the AF-algebra $\mathscr{A}$ is called 
stationary.  In particular, such an AF-algebra is given by a constant positive
integer matrix  $A\in GL(n, \mathbf{Z})$.

\subsubsection{K-theory}
By $M_{\infty}(\mathscr{A})$ 
one understands the algebraic direct limit of the $C^*$-algebras 
$M_n(\mathscr{A})$ under the embeddings $a\mapsto ~\mathbf{diag} (a,0)$. 
The direct limit $M_{\infty}(\mathscr{A})$  can be thought of as the $C^*$-algebra 
of infinite-dimensional matrices whose entries are all zero except for a finite number of the
non-zero entries taken from the $C^*$-algebra $\mathscr{A}$.
Two projections $p,q\in M_{\infty}(\mathscr{A})$ are equivalent, if there exists 
an element $v\in M_{\infty}(\mathscr{A})$,  such that $p=v^*v$ and $q=vv^*$. 
The equivalence class of projection $p$ is denoted by $[p]$.   
We write $V(\mathscr{A})$ to denote all equivalence classes of 
projections in the $C^*$-algebra $M_{\infty}(\mathscr{A})$, i.e.
$V(\mathscr{A}):=\{[p] ~:~ p=p^*=p^2\in M_{\infty}(\mathscr{A})\}$. 
The set $V(\mathscr{A})$ has the natural structure of an abelian 
semi-group with the addition operation defined by the formula 
$[p]+[q]:=\mathbf{diag}(p,q)=[p'\oplus q']$, where $p'\sim p, ~q'\sim q$ 
and $p'\perp q'$.  The identity of the semi-group $V(\mathscr{A})$ 
is given by $[0]$, where $0$ is the zero projection. 
By the $K_0$-group $K_0(\mathscr{A})$ of the unital $C^*$-algebra $\mathscr{A}$
one understands the Grothendieck group of the abelian semi-group
$V(\mathscr{A})$, i.e. a completion of $V(\mathscr{A})$ by the formal elements
$[p]-[q]$.

The canonical trace $\tau$ on the AF-algebra  $\mathbb{A}$. 
 induces a homomorphism $\tau_*: K_0(\mathbb{A})\to\mathbf{R}$. 
We let $\mathfrak{m}:=\tau_*( K_0(\mathbb{A}))\subset\mathbf{R}$.  
If    $\mathbb{A}$   is the stationary AF-algebra given by a matrix $A\in GL(n, \mathbf{Z})$, 
then $\mathfrak{m}$ is a $\mathbf{Z}$-module 
in the number field $k=\mathbf{Q}(\lambda_A)$, where  $\lambda_A>1$ is  the Perron-Frobenius 
eigenvalue of  the matrix $A$.    The endomorphism ring of $\mathfrak{m}$
is denoted by $\Lambda$ and the ideal class of $\mathfrak{m}$ is denoted by $[\mathfrak{m}]$. 
The $(\Lambda, [\mathfrak{m}], k)$ is called a Handelman triple of the stationary AF-algebra 
 $\mathbb{A}$.

\subsubsection{Etesi $C^*$-algebra}
Let $\mathscr{M}$ be a smooth 4-dimensional manifold. 
The  group $C^*$-algebra  $\mathbb{E}_{\mathscr{M}}$ of the locally compact group 
$\mathit{Diff}(\mathscr{M})/ \mathit{Diff}_0(\mathscr{M})$
is called the Etesi $C^*$-algebra of $\mathscr{M}$. 
Some properties of  the $\mathbb{E}_{\mathscr{M}}$ are described in below. 
\begin{theorem}\label{cor1.4}
{\bf (\cite[Section 7.5]{Nik1})}
The following is true:

\medskip
(i)  the $\mathbb{E}_{\mathscr{M}}$ is a stationary AF-algebra;

\smallskip
(ii)  the Handelman triple $(\Lambda, [\mathfrak{m}], k)$ is a topological
invariant  of $\mathscr{M}$. 
\end{theorem}
\begin{remark}\label{rmk2.2}
The map $F$ acting by the formula:
\begin{equation}\label{eq2.1}
\mathscr{M}\mapsto (\Lambda, [\mathfrak{m}], k)
\end{equation}
is a covariant functor on the category of all 4-dimensional manifolds
with values in a category of the Handelman triples.
In particular, the number field $k$ is a topological invariant of $\mathscr{M}$. 
It is not hard to see, that (\ref{eq2.1}) has an inverse $(\Lambda, [\mathfrak{m}], k)
\mapsto \mathscr{M}$.  Indeed, the group algebra  $\mathbb{E}_{\mathscr{M}}$
can be recovered from the triple $(\Lambda, [\mathfrak{m}], k)$,  while 
the manifold $\mathscr{M}$ is defined by  the group 
$\mathit{Diff}(\mathscr{M})/ \mathit{Diff}_0(\mathscr{M})$. 
\end{remark}

\subsection{Algebraic surfaces}
An algebraic surface is a variety $S$ of the complex dimension $2$. 
One can identify the non-singular variety $S$ with a  complex surface and therefore
with a smooth 4-dimensional manifold $\mathscr{M}$.  
In what follows, we denote by $\mathbb{E}_S$ the Etesi 
$C^*$-algebra of $\mathscr{M}$ corresponding 
to the surface $S$.

The map $\phi: S\dashrightarrow S'$ is called rational,  if it is given  
by a rational function. 
The rational maps cannot be composed unless they are dominant, i.e. the image 
of $\phi$ is Zariski dense in $S'$. 
The map $\phi$ is birational,  if the inverse $\phi^{-1}$ is a rational map. 
A birational map  $\epsilon: S\dashrightarrow S'$ is called a blow-up,
if it is defined everywhere except for a point $p\in S$ and a rational curve $C\subset S'$,
such  that  $\epsilon^{-1}(C)=p$.   
Every birational map $\phi: S\dashrightarrow S'$ is composition of a finite 
number of the blow-ups, i.e. $\phi=\epsilon_1\circ\dots\circ\epsilon_k$.

The surface $S$ is called a minimal model, if any birational map $S\dashrightarrow S'$
is an isomorphism. The minimal models exist and are unique unless $S$ is 
a ruled surface. By the Castelnuovo Theorem, the smooth projective surface $S$ is a minimal 
model if and only if $S$ does not contain a rational curves $C$ with the 
self-intersection index $-1$.

\section{Proofs}
\subsection{Proof of theorem \ref{thm1.1}}
For the sake of clarity, let us outline the main ideas.
Let $S\to S'$ be a regular (polynomial) map between the surfaces
$S$ and $S'$. It is known that such a map induces a homomorphism  $\mathbb{E}_S\to \mathbb{E}_{S'}$
of the  Etesi  $C^*$-algebras and an extension of the fields $k\subseteq k'$
in the corresponding Handelman triples  $(\Lambda, [\mathfrak{m}], k)\subseteq
 (\Lambda', [\mathfrak{m}'], k')$, see Section 2.2.3. 
Recall that the rational map  $S\dashrightarrow S'$ is regular only on an open subset $U\subset S$,
such that the Zariski closure of $U$ coincides with $S$. 
Roughly speaking, we prove that an operation corresponding to the Zariski closure of $U$ 
consists in passing from the field $k$ to its Hilbert class field 
$\mathscr{H}(k)$  (lemma \ref{lm3.1}).  
The rest of the proof follows from the 
inclusion of fields  $\mathscr{H}(k)\subseteq \mathscr{H}(k')$ induced by the rational 
map  $S\dashrightarrow S'$. 
 
\begin{figure}
\begin{picture}(300,110)(-70,0)
\put(20,70){\vector(0,-1){35}}
\put(122,70){\vector(0,-1){35}}
\put(45,23){\vector(1,0){60}}
\put(45,83){\vector(1,0){60}}
\put(15,20){$k_0$}
\put(118,20){$k$}
\put(17,80){$S_0$}
\put(115,80){ $S$}
\put(55,30){\sf inclusion}
\put(50,90){\sf  regular map}
\end{picture}
\caption{}
\end{figure}

\bigskip
We shall split the proof in a series of lemmas. 
 \begin{lemma}\label{lm3.1}
 If $S\dashrightarrow S'$ is a rational map, then $\mathscr{H}(k)\subseteq \mathscr{H}(k')$.
\end{lemma}
\begin{proof}
In outline, 
an open set $U\subset S$ is an open 4-dimensional manifold 
with boundary. Taking a connected sum with the copies of $S^4$, one gets a compact smooth 
manifold $S_0$ and  a regular map $S_0\to S$. Such a map defines a field extension $k_0\subseteq k$. 
Since $U$ is Zariski dense in $S$, we conclude that 
the surface $S_0$ determines $S$ up to an isomorphism.  
Therefore the field extension $k_0\subseteq k$
must depend solely on the arithmetic of the field $k_0$. 
In other words, the intrinsic  invariants of $k_0$  control the Galois group $Gal ~(k|k_0)$.
This can happen if and only if $k\cong \mathscr{H}(k_0)$, so that    
 $Gal ~(k|k_0)\cong Cl~(k_0)$,  where   $Cl~(k_0)$  is the ideal class group of $k_0$. 
We pass to a detailed argument. 

\bigskip
(i)  Let  $\phi: S\dashrightarrow S'$ be a rational map. 
Then there exist the open sets  $U\subset S$ and $U'\subset S'$,
such that 
\begin{equation}\label{eq3.1}
\phi: U\longrightarrow U'. 
\end{equation}
is a regular map. 

\bigskip
(ii) Since $U\subset S$ is an open set, it is Zariski dense in $S$. 
The set $U$ is a  4-dimensional manifold with a boundary
corresponding to the poles of the rational map  $S\dashrightarrow S'$.  Let $n$ be the total number of the boundary 
components of $U$. Consider a compact smooth 4-dimensional manifold 
\begin{equation}\label{eq3.2}
S_0:=U ~\#_n ~S^4.
\end{equation}
 coming from the connected sum of 
$U$ with  the $n$ copies of the 4-dimensional sphere $S^4$ equipped 
with the standard smooth structure.

\bigskip
(iii) Notice  that $S_0$ can be endowed with a complex structure and can be identified 
with  an algebraic surface.  It is not hard to see, that there exists a regular map $S_0\to S$
and the corresponding field extension $k_0\subseteq k$ given by the commutative diagram  in Figure 1. 
Since  $k_0$ and $k$ are totally real number fields,  the $k_0\subseteq k$ is a Galois extension.
In particular, the Galois group  $Gal~(k|k_0)$ is correctly defined.

\bigskip
(iv) 
Recall that the Zariski closure of $U$ coincides with the surface $S$. 
Since  $S_0$ contains $U$, the Zariski closure of $S_0$ will coincide with $S$
as well.  Using  the diagram in Figure 1, we  conclude
that the Galois extension   $k_0\subseteq k$ depends only on the arithmetic of the ground
field $k_0$.  This means that the Galois group $Gal~(k|k_0)$ must be an invariant of
the field $k_0$. The only extension with such a property is the Hilbert class field 
$\mathscr{H}(k_0)$, i.e.     $Gal~(k|k_0)\cong Cl~(k_0)$, where $Cl~(k_0)$
is the ideal class group of $k_0$.  Thus $k\cong \mathscr{H}(k_0)$.

\bigskip
(v)  Using the regular map (\ref{eq3.1}), one gets an inclusion of the number fields 
$k_0\subseteq k_0'$ and therefore an inclusion $\mathscr{H}(k_0)\subseteq \mathscr{H}(k_0')$. 
In other words, the diagram in Figure 1 implies a commutative diagram in Figure 2.

\begin{figure}
\begin{picture}(300,110)(-70,0)
\put(20,70){\vector(0,-1){35}}
\put(122,70){\vector(0,-1){35}}
\put(45,23){\vector(1,0){60}}
\put(45,83){\vector(1,0){60}}
\put(10,20){$\mathscr{H}(k_0)$}
\put(113,20){$\mathscr{H}(k_0')$}
\put(17,80){$S$}
\put(115,80){ $S'$}
\put(55,30){\sf inclusion}
\put(50,90){\sf  rational map}
\end{picture}
\caption{}
\end{figure}

\bigskip
(vi) Lemma \ref{lm3.1} follows from Figure 2 after an adjustment of  the notation,
i.e. dropping  the subscript zero for the number field $k_0$.  
\end{proof}

 \begin{lemma}\label{lm3.2}
 If $S\dashrightarrow S'$ is a birational map, then $\mathscr{H}(k)\cong \mathscr{H}(k')$.
 \end{lemma}
\begin{proof}
In view of lemma \ref{lm3.1},  the rational map $S\dashrightarrow S'$ implies an inclusion 
of the number fields $\mathscr{H}(k)\subseteq\mathscr{H}(k')$. 
Since $S\dashrightarrow S'$ is birational,  the inverse rational map 
$S'\dashrightarrow S$ gives an inclusion of the number fields $\mathscr{H}(k')\subseteq\mathscr{H}(k)$.
Clearly, the above inclusions are compatible if and only if  $\mathscr{H}(k')\cong\mathscr{H}(k)$. 
Lemma \ref{lm3.2} is proved. 
\end{proof} 

 \begin{lemma}\label{lm3.3}
If  $S\dashrightarrow S'$ is a blow-up, then 
 $k'\cong \mathscr{H}(k)$.
 \end{lemma}
\begin{proof}
Roughly speaking, we use the same argument as in lemma \ref{lm3.1}. 
Since the blow-up is a dominant rational map, the image of surface $S$ must be Zariski dense in $S'$. 
As it was shown earlier, one gets  an isomorphism between the 
field $k'$ and the Hilbert class field of $k$. We pass to a detailed argument. 

\bigskip
(i) Recall that any birational map $\phi: S\dashrightarrow S'$ is a composition
\begin{equation}\label{eq3.3bis}
\phi=\epsilon_1\circ\dots\circ\epsilon_m,
\end{equation}
where $\epsilon_i$ is a blow-up and $m<\infty$. 
In particular, the $\epsilon_i$ must be dominant rational maps. The latter means that  the image 
$\epsilon_i(S)$ is Zariski dense in $S'$.

\bigskip
(ii)   Consider a dominant rational map   $\phi: S\dashrightarrow S'$ and the corresponding
Galois  extension of the number fields $k\subseteq k'$ shown in Figure 3. 
Since  $S'$ is  the closure of a Zariski  dense subset $\phi(S)$,  we conclude 
that the extension   $k\subseteq k'$ depends only on the arithmetic of the ground
field $k$.  In particular,  the Galois group $Gal~(k'|k)$ is an invariant of
the field $k$. The only extension with such a property is the Hilbert class field 
$\mathscr{H}(k)$ for which   $Gal~(k|k_0)\cong Cl~(k_0)$, 
where $Cl~(k_0)$ is the ideal class group of $k_0$.   Thus $k'\cong \mathscr{H}(k)$. 
Lemma \ref{lm3.3} is proved. 
\end{proof}

\begin{figure}
\begin{picture}(300,110)(-70,0)
\put(20,70){\vector(0,-1){35}}
\put(122,70){\vector(0,-1){35}}
\put(45,23){\vector(1,0){60}}
\put(45,83){\vector(1,0){60}}
\put(18,20){$k$}
\put(118,20){$k'$}
\put(17,80){$S$}
\put(115,80){ $S'$}
\put(55,30){\sf inclusion}
\put(45,90){\sf  dominant map}
\end{picture}
\caption{}
\end{figure}

\bigskip
The `if'  part of theorem \ref{thm1.1} follows from lemma \ref{lm3.3}. Let us show
that if $k'\cong \mathscr{H}(k)$,  then  the corresponding birational map 
$\phi: S\dashrightarrow S'$  is a blow-up. Indeed,  since $\phi$ is a birational map,
one can apply lemma \ref{lm3.2}  to obtain an isomorphism of the number fields
\begin{equation}\label{eq3.3}
\mathscr{H}(k)\cong \mathscr{H}(k'). 
\end{equation}

\bigskip
The substitution  $k'\cong \mathscr{H}(k)$ into (\ref{eq3.3}) will imply
that $\mathscr{H}(k)\cong \mathscr{H}^2(k)$.  In other words, the class field tower of $k$ 
is stable after the first step.    In particular, such a tower cannot be decomposed into the 
sub-towers, i.e. the birational map $\phi$ cannot be decomposed into a composition of 
the blow-ups $\epsilon_i$. The latter means that $\phi$ is a blow-up itself, see (\ref{eq3.3bis}).  
The `only if' part of theorem \ref{thm1.1} follows. 

\bigskip
This argument finishes the proof of theorem \ref{thm1.1}.

\subsection{Proof of corollary \ref{cor1.2}}
Our proof is based on the Castelnuovo theory of the minimal models for algebraic 
surfaces.    Namely,  such a theory says that it takes a finite number of the blow-ups to 
get the minimal model of $S$.  The rest of the proof follows from 
theorem \ref{thm1.1} applied to the corresponding class field tower. We pass to a detailed 
argument. 

\bigskip
(i)  Let us prove that if $S$ is an algebraic surface, then its minimal model gives rise to  a finite real 
class field tower.  We denote by $S^{(m)}$ the minimal model obtained from $S$ by  composition 
of the blow-ups 
\begin{equation}\label{eq3.4}
\epsilon_i: ~S^{(i-1)}\dashrightarrow S^{(i)}, 
\end{equation}
where  $1\le i\le m$. 
In view of theorem \ref{thm1.1},  each blow-up $\epsilon_i$ defines a Hilbert class field extension
of  the $k_{i-1}$, i.e.
\begin{equation}\label{eq3.5}
k_i=\mathscr{H}(k_{i-1}), 
\end{equation}
where all $\{k_i ~|~1\le i\le m\}$ are real number fields.

\bigskip
(ii) After a finite number of the blow-ups (\ref{eq3.4}),  one gets a commutative diagram in Figure 4. 
Thus the finite real  class field tower corresponding to the minimal model $S^{(m)}$ has the form:
\begin{equation}\label{eq3.6}
k\subset \mathscr{H}(k)\subset \mathscr{H}^2(k)\subset\dots\subset 
\mathscr{H}^{m}(k) \cong \mathscr{H}^{m+1}(k)\cong\dots
\end{equation}

\begin{figure}
\begin{picture}(300,110)(0,0)

\put(20,70){\vector(0,-1){35}}
\put(55,70){\vector(0,-1){35}}
\put(92,70){\vector(0,-1){35}}
\put(165,70){\vector(0,-1){35}}
\put(18,20){$k \hookrightarrow  \mathscr{H}(k) \hookrightarrow  \mathscr{H}^2(k)
 \hookrightarrow\dots \hookrightarrow\mathscr{H}^{m}(k) \cong \mathscr{H}^{m+1}(k)\cong\dots$}
\put(17,80){$S ~\dashrightarrow ~S' 
~\dashrightarrow ~S''\dashrightarrow ~\dots ~\dashrightarrow S^{(m)}$}
\end{picture}
\caption{}
\end{figure}

\bigskip
(iii)  Let us show that if 
\begin{equation}\label{eq3.7}
k:=\mathscr{H}^0(k)\subset \mathscr{H}^1(k)\subset \mathscr{H}^2(k)\subset\dots
\end{equation}
is a real class field tower, then it is finite, i.e. 
there exists an integer $m<\infty$ such that 
$\mathscr{H}^{m}(k) \cong \mathscr{H}^{m+1}(k)\cong\dots$. 
Indeed, since $k_i:=\{\mathscr{H}^i(k) ~|~i\ge 0\}$ is a real number
field,  one can construct a Handelman triple  $(\Lambda_i, [\mathfrak{m}_i], k_i)$ 
and a smooth 4-dimensional manifold $\mathscr{M}_i$, see remark \ref{rmk2.2}.  
Since the triple $(\Lambda_i, [\mathfrak{m}_i], k_i)$  is a topological invariant 
of $\mathscr{M}_i$,  we can always assume that  $\mathscr{M}_i$ is a complex
surface by choosing a proper smoothing of $\mathscr{M}_i$ if necessary. 
Notice that one can identify $\mathscr{M}_i$ with an algebraic
surface $S^{(i)}$.

\bigskip
(iv) Since $k_i=\mathscr{H}(k_{i-1})$, one can apply  theorem \ref{thm1.1} saying  that
 the $S^{(i)}$ is a blow-up of the surface  $S^{(i-1)}$.  Using 
the Castelnuovo theorem, one concludes that 
the class field tower (\ref{eq3.7})  must stabilize for an integer $m<\infty$,
i.e.  such a the tower is always finite. 

\bigskip
This argument finishes the proof of corollary \ref{cor1.2}.


\bibliographystyle{amsplain}


\end{document}